\newcommand{\AMSclass}[1]{{\textbf{AMS subject classification:} #1}}
\newcommand{\email}[1]{{\textit{Email:} \texttt{#1}}}
\newcommand{\keywords}[1]{{\textbf{Keywords:} #1}}

\documentclass{article}
\usepackage{enumerate,bbm,amsthm}
\usepackage[T1]{fontenc}
\usepackage[a4paper]{geometry}
\usepackage{amsmath}
\usepackage{amsfonts}
\usepackage{amssymb}
\usepackage{mathptmx}
\newcommand{\assign}{:=}
\newcommand{\comma}{{,}}

\newcommand{\nocomma}{}

\newcommand{\tmem}[1]{{\em #1\/}}

\newcommand{\tmop}[1]{\ensuremath{\operatorname{#1}}}

\newtheorem{lemma}{Lemma}
\newtheorem{proposition}{Proposition}
\newtheorem{theorem}{Theorem}
\newtheorem{corollary}{Corollary}
\theoremstyle{remark}
\newtheorem*{acknowledgement}{Acknowledgement}
\newtheorem*{remark}{Remark}

\begin{document}
\title{The jump of the Milnor number in the $X_{9}$ singularity class
\thanks{\AMSclass{32S05, 14B05, 32S30, 14B07}, \keywords{Milnor number, singularity, deformation of singularity}}}
\author{Szymon Brzostowski and Tadeusz Krasi\'nski
\thanks{\email{brzosts@math.uni.lodz.pl}, \email{krasinsk@uni.lodz.pl}}\\
\\
\small Faculty of Mathematics and Computer Science,\\
\small University of \L\'od\'z,\\
\small ul. Banacha 22, 90-238 \L\'od\'z, Poland}
\maketitle
\begin{abstract}
The jump of the Milnor number of an isolated singularity $f_{0} $ is the
minimal non-zero difference between the Milnor numbers of $f_{0}$ and one of
its deformations $\left(  f_{s} \right)  $. We prove that for the
singularities in the $X_{9}$ singularity class their jumps are equal
to $2$.
\end{abstract}

\section{Introduction}

\label{Par 1}

Let $f_{0} : (\mathbbm{C}^{n}, 0) \rightarrow(\mathbbm{C}, 0)$ be an
{\emph{(isolated) singularity\/}}, i.e. $f_{0}$ is a germ at $0$ of a
holomorphic function having an isolated critical point at $0 \in
\mathbbm{C}^{n}$, and $0 \in\mathbbm{C}$ as the corresponding critical value.
More specifically, there exists a representative $\hat{f}_{0} : U
\rightarrow\mathbbm{C}$ of $f_{0} $, holomorphic in an open neighborhood $U$
of the point $0 \in\mathbbm{C}^{n}$, such that $\hat{f}_{0} \left(  0 \right)  = 0$,
 $\nabla\hat{f}_{0} \left(  0 \right)  = 0$ and $\nabla\hat{f}_{0} \left(  z \right)  \neq0$ for $z \in U\backslash
\left\{  0 \right\}  $, where for a holomorphic function $f$ we put $\nabla f = \nabla_{z}
f := \left(  \partial f / \partial z_{1}, \ldots, \partial f / \partial z_{n}
\right)  $.

In the sequel we will identify germs of holomorphic functions with their
representatives or the corresponding convergent power series. The ring of
germs of holomorphic functions of $n$ variables will be denoted by
$\mathcal{O}^{n}$.

A {\emph{deformation of the singularity\/}} $f_{0}$ is the germ of a
holomorphic function $f = f \left(  s, z \right)  : \left(  \mathbbm{C}
\times\mathbbm{C}^{n}, 0 \right)  \rightarrow\left(  \mathbbm{C}, 0 \right)  $
such that:

\begin{enumerate}
\item $f \left(  0, z \right)  = f_{0} \left(  z \right)  $,

\item $f \left(  s, 0 \right)  = 0$,

\item for each $\left|  s \right|  \ll1$ it is $\nabla_{z} f \left(  s, z
\right)  \neq0$ for $z \neq0$ in a (small) neighborhood of $0 \in
\mathbbm{C}^{n}$.
\end{enumerate}

The deformation $f \left(  s, z \right)  $ of the singularity $f_{0}$ will
also be treated as a family $\left(  f_{s} \right)  $ of germs, taking $f_{s}
\left(  z \right)  := f \left(  s, z \right)  $. In this context, the symbol
$\nabla f_{s}$ will always denote $\nabla_{z} f_{s}$.

{\remark{\label{nota1}Notice that in the deformation $\left( f_s \right)$ of
$f_0$ there can occur {\tmem{smooth}} germs, that is germs satisfying $\nabla
f_s \left( 0 \right) \neq 0$.}}{\medskip}

By the above assumptions it follows that, for every sufficiently small $s$,
one can define a (finite) number $\mu_{s}$ as the Milnor number of $f_{s}$,
namely
\[ \mu_s \assign \mu \left( f_s \right) \mathrel{=} \dim_{\mathbbm{C}} 
   \mathbin{\mathbin{\mathcal{O}}^n / \mathbin{\left( \nabla f_s \right)}}
   \mathord{\mathrel{}} \mathrel{=} i_0 \left( \frac{\partial
   f_s}{\partial z_1}, \ldots, \frac{\partial f_s}{\partial z_n} \right) \text{,}
\]
where the symbol $i_0 \left( \frac{\partial f_s}{\partial z_1},
\ldots, \frac{\partial f_s}{\partial z_n} \right)$ denotes the multiplicity of
the ideal $\left( \frac{\partial f_s}{\partial z_1}, \ldots, \frac{\partial
f_s}{\partial z_n} \right) \mathcal{O}^n$.
Since the Milnor number is upper semi-continuous in the Zariski topology in families of singularities
{\cite[Ch. I, Thm. 2.6 and Ch. II, Prop. 2.57]{GLS07}}, there exists an open neighborhood $S$ of the
point $0 \in\mathbbm{C}$ such that

\begin{enumerate}
\item $\mu_{s}=\operatorname{const.}$ for $s\in S\setminus\left\{
0\right\}  $,

\item $\mu_{0} \geqslant\mu_{s}$ for $s \in S$.
\end{enumerate}

The (constant) difference $\mu_{0} - \mu_{s}$ for $s \in S \setminus\left\{  0
\right\}  $ will be called {\emph{the jump of the deformation $\left(  f_{s}
\right)  $\/}} and denoted by $\lambda\left(  \left(  f_{s} \right)  \right)
$. The smallest nonzero value among all the jumps of deformations of the
singularity $f_{0}$ will be called {\emph{the jump (of the Milnor number) of
the singularity $f_{0}$\/}} and denoted by $\lambda\left(  f_{0} \right)  $.

The first general result concerning the problem of computation of the jump was
S. Gusein-Zade's {\cite{Gus93}}, who proved that there exist singularities
$f_{0}$ for which $\lambda\left(  f_{0} \right)  > 1$ and that for irreducible
plane curve singularities $f_{0}$ it holds $\lambda\left(  f_{0} \right)  =
1$. He showed that generic elements in some classes of singularities
(satisfying conditions concerning the Milnor numbers and modality) fulfill
$\lambda\left(  f_{0} \right)  > 1$, but he did not give any specific example
of such a singularity.

The two-dimensional version of the problem of computation of the jump, and
more precisely -- of {\emph{the non-degenerate jump}\/} (i.e. all the families
$\left(  f_{s}\right)  $ being considered are to be made of Kouchnirenko
non-degenerate singularities), has been studied in
{\cite{Bod07}}, {\cite{Wal10}}. 

The following are examples of classes of singularities that fulfill the
assumptions of the Gusein-Zade theorem.

\begin{enumerate}
\item The class $X_{9}$, in the terminology of {\cite{AGV85}. It consists of
singularities stably equivalent to the singularities} of the form $f_{0}%
^{a}(x,y):=x^{4}+y^{4}+ax^{2}y^{2},\text{{}}a\in\mathbbm{C},a^{2}\neq4$. The
singularities are of modality $1$ and $\mu(f_{0}^{a})=9$.

\item The class $W_{1,0}$, in the terminology of {\cite{AGV85}. It consists of
singularities stably equivalent to the singularities} of the form
$f_{0}^{(a,b)}(x,y):=x^{4}+y^{6}+\left(  a+by\right)  x^{2}y^{3},\text{{}%
}a,b\in\mathbbm{C},a^{2}\neq4$. The singularities are of modality $2$ and
$\mu(f_{0}^{(a,b)})=15$.
\end{enumerate}

By the Gusein-Zade result, generic elements $f$ of the classes $X_{9}$ and
$W_{1,0}$ satisfy $\lambda\left(  f\right)  >1$. However,
determining the jump of any particular element of these classes is still an
open and difficult problem. Gusein-Zade did not give any specific
example of a singularity $f$ with $\lambda\left(  f\right)  >1$. The purpose
of this work is to prove (Thm. \ref{Th1}) that for the singularities in the $X_{9}$ class we have
\[
\lambda(  f_{0}^a)  =2
\]
(and that therefore all the singularities of the class $X_{9}$ are
,,generic\textquotedblright\ in the family $X_{9}$).
In the class $W_{1,0}$ we obtain only a partial result (Prop. \ref{Th0}). Namely, for the
singularities in $W_{1,0}$ that are stably equivalent to the ones in the subclass%
\[
f_{0}^{\left(0,b\right)}\left(  x,y\right)  =x^{4}+y^{6}+bx^{2}y^{4}%
{,}\hspace{1em}b\in\mathbbm{C}{,}
\]
we have
\[
\lambda(  f_{0}^{\left(  0,b\right)  })  =1
\]
(therefore these singularities are not ,,generic\textquotedblright%
\ in the family $W_{1,0}$).

This implies that the jump $\lambda(f_0)$ is not a topological
invariant of singularities (Cor. \ref{cor2}). {\medskip}

In the light of the above results the following problems arise:
\begin{enumerate}
\item Show that for the remaining singularities in the $W_{1,0}$ class, i.e.
for the singularities stably equivalent to $f_{0}^{\left(  a,b\right)
}:=x^{4}+y^{6}+\left(  a+by\right)  x^{2}y^{3}$, where $a,b\in
\mathbbm{C},0\neq a^{2}\neq4$, we have $\lambda(  f_{0}^{\left(
a,b\right)  })  =2$,
\end{enumerate}

\noindent and more general ones (posed by Bodin in {\cite{Bod07}}):

\begin{enumerate}
\item[(2)] Find an algorithm that computes $\lambda\left(  f_{0}\right)  $.

\item[(3)] Give the list of all possible Milnor numbers arising from
deformations of $f_{0}$ (see {\cite{Wal10}} for partial results in the
non-degenerate case).
\end{enumerate}

\section{Preliminaries}

\label{se2}

Let $\mathbbm{N}$ be the set of nonnegative
integers and $\mathbbm{R}_{+}$ be the set of nonnegative real
numbers. Let $f_{0}\left(  x,y\right)  =\sum_{\left(  i,j\right)
\in\mathbbm{N}^{2}}a_{ij}x^{i}y^{j}$ be a singularity. Put
$\operatorname{supp}\left(  f_{0}\right)  :=\left\{  \left(  i,j\right)
\in\mathbbm{N}^{2}:a_{ij}\neq0\right\}  $. The {\emph{Newton diagram of
$f_{0}$}\/} is defined as the convex hull of the set
\[
\bigcup_{\left(  i,j\right)  \in\operatorname{supp}\left(  f_{0}\right)
}\left(  i,j\right)  +\mathbbm{R}_{+}^{2}%
\]
and is denoted by $\Gamma_{+}\left(  f_{0}\right)  $. It is easy to see that
the boundary (in $\mathbbm{R}^{2}$) of the diagram $\Gamma_{+}\left(
f_{0}\right)  $ is a sum of two half-lines and a finite number of compact line
segments. The set of those line
segments will be called the {\emph{Newton polygon of the singularity $f_{0}$%
}\/} and denoted by $\Gamma\left(  f_{0}\right)  $. For each segment
$\gamma\in\Gamma\left(  f_{0}\right)  $ we define a weighted homogeneous
polynomial
\[
\left(  f_{0}\right)  _{\gamma}:=\sum_{\left(  i,j\right)  \in\gamma}%
a_{ij}x^{i}y^{j}.
\]

A singularity $f_{0}$ is called {\emph{non-degenerate (in the Kouchnirenko
sense) on a segment $\gamma\in\Gamma\left(  f_{0}\right)  $}\/} iff the
system
\[
\frac{\partial\left(  f_{0}\right)  _{\gamma}}{\partial x}\left(  x,y\right)
=0,\frac{\partial\left(  f_{0}\right)  _{\gamma}}{\partial y}\left(
x,y\right)  =0
\]
has no solutions in $\mathbbm{C}^{\ast}\times\mathbbm{C}^{\ast}$. $f_{0}$ is
called {\emph{non-degenerate}\/} iff it is non-degenerate on every segment
$\gamma\in\Gamma\left(  f_{0}\right)  $.

\label{Punkt 2.3}

For the sake of simplicity, we consider the case of {\emph{convenient}\/} singularities $f_{0}$, i.e. we
suppose that $\Gamma_{+}\left(  f_{0}\right)  $ intersects both coordinate axes
in $\mathbbm{R}^{2}$. For such
singularities we denote by $S$ the area of the domain bounded by the
coordinate axes and the Newton polygon $\Gamma\left(  f_{0}\right)  $. Let
$a$ (resp. $b$) be the distance of the point
$\left(  0,0\right)  $ to the intersection of $\Gamma_{+}\left(  f_{0}\right)
$ with the horizontal (resp. vertical) axis. The number
\[
\nu\left(  f_{0}\right)  :=2S-a-b+1
\]
is called the {\emph{Newton number of the singularity $f_{0}$}\/}. Let us recall
Planar Kouchnirenko Theorem.

\begin{theorem}
\label{Th Kouch}{\textup{\textbf{({\cite{Kou76}}) }}} For a
convenient singularity $f_{0}$ we have:

\begin{enumerate}
\item $\mu(f_{0}) \geqslant\nu(f_{0})$,

\item if $f_{0}$ is non-degenerate then $\mu(f_{0}) = \nu(f_{0})$.
\end{enumerate}
\end{theorem}
Theorem \ref{Th Kouch} can be completed in the following way.
\begin{theorem}
\label{Th Plo}{\textup{\textbf{(P{\l }oski, {\cite{Plo90,Plo99}}) }}}If for a
convenient singularity $f_{0}$ there is $\nu(f_{0}) = \mu(f_{0})$ then $f_{0}$ is non-degenerate.
\end{theorem}

We will also need a ,,global'' result concerning projective algebraic curves.

\begin{theorem}
\label{Th GP}{\textup{\textbf{({\cite[Prop. 6.3]{GP01}}) }}}Let $\mathcal{C}
\subset\mathbbm{C}\mathbbm{P}^{2}$ be a projective algebraic curve of degree
$d$. Suppose that $m$ irreducible components of $\mathcal{C}$ pass through a
point $P \in\mathcal{C}$. Then the Milnor number $\mu_{P} \left(  \mathcal{C}
\right)  $ of $\mathcal{C}$ at $P$ satisfies the inequality
\[
\mu_{P} \left(  \mathcal{C} \right)  \leqslant\left(  d - 1 \right)  \left(  d
- 2 \right)  + m - 1.
\]

\end{theorem}

The rest of the section is devoted mainly to the concept of a versal
unfolding. It is based on the book by Ebeling {\cite{Ebe07}}.

Let $f_{0} : \left(  \mathbbm{C}^{n}, 0 \right)  \rightarrow\left(
\mathbbm{C}, 0 \right)  $ be a germ of a holomorphic function. An
{\emph{unfolding of $f_{0}$\/}} is a holomorphic germ $F : \left(
\mathbbm{C}^{n} \times\mathbbm{C}^{k}, 0 \right)  \rightarrow\left(
\mathbbm{C}, 0 \right)  $ such that $F \left(  z, 0 \right)  = f_{0} \left(  z
\right)  $ and $F \left(  0, u \right)  = 0$.

Two unfoldings $F : \left(  \mathbbm{C}^{n} \times\mathbbm{C}^{k}, 0 \right)
\rightarrow\left(  \mathbbm{C}, 0 \right)  $ and $G : \left(  \mathbbm{C}^{n}
\times\mathbbm{C}^{k}, 0 \right)  \rightarrow\left(  \mathbbm{C}, 0 \right)  $
of $f_{0}$ are said to be {\emph{equivalent\/}}, if there exists a holomorphic
map-germ
\[
\psi: \left(  \mathbbm{C}^{n} \times\mathbbm{C}^{k}, 0 \right)  \rightarrow
\left(  \mathbbm{C}^{n}, 0 \right)  , \hspace{1em} \psi\left(  z, 0 \right)  =
z, \hspace{1em} \psi\left(  0, u \right)  = 0
\]
{\noindent}such that
\[
G \left(  z, u \right)  = F \left(  \psi\left(  z, u \right)  , u \right)  .
\]
It is easy to see that this notion of equivalence is in fact an equivalence
relation in the set of unfoldings of $f_{0}$.

Let $F : \left(  \mathbbm{C}^{n} \times\mathbbm{C}^{k}, 0 \right)
\rightarrow\left(  \mathbbm{C}, 0 \right)  $ be an unfolding of $f_{0}$ and
$\varphi: \left(  \mathbbm{C}^{l}, 0 \right)  \rightarrow\left(
\mathbbm{C}^{k}, 0 \right)  $ -- a holomorphic map-germ. The {\emph{unfolding
of $f_{0}$ induced from $F$ by $\varphi$\/}} is defined by the formula
\[
G \left(  z, u \right)  = F \left(  z, \varphi\left(  u \right)  \right)  .
\]
An unfolding $F : \left(  \mathbbm{C}^{n} \times\mathbbm{C}^{k}, 0 \right)
\rightarrow\left(  \mathbbm{C}, 0 \right)  $ of $f_{0}$ is called
{\emph{versal\/}} if any unfolding of $f_{0}$ is equivalent to one induced
from $F$.

The following proposition will be useful.

\begin{proposition}
\label{Pro 2}{\textup{\textbf{({\cite[Ch. 4, Prop. 2.4]{Mar82}}) }}}If $f
\in\mathcal{O}^{n}$ is a singularity, $\mathfrak{m}$ is the maximal ideal in
$\mathcal{O}^{n}$, then
\[
\dim_{\mathbbm{C}} \dfrac{\mathcal{O}^{n}}{\mathfrak{m} \left(  \nabla f
\right)  \mathcal{O}^{n}} = \dim_{\mathbbm{C}} \frac{\mathcal{O}^{n}}{\left(
\nabla f \right)  \mathcal{O}^{n}} + n \text{.}%
\]

\end{proposition}

The main result concerning versal unfoldings is the following.

\begin{theorem}
\label{Th Ebe}Let $f_{0} : \left(  \mathbbm{C}^{n}, 0 \right)  \rightarrow
\left(  \mathbbm{C}, 0 \right)  $ be a singularity and put $\mu= \mu\left(
f_{0} \right)  $. Let $g_{1}, \ldots, g_{\mu+ n - 1} \in\mathcal{O}^{n}$ be
any representatives of a basis of the $\mathbbm{C}$--vector space
$\frac{\mathfrak{m}}{\mathfrak{m} \left(  \nabla f_{0} \right)  }$. Then the
holomorphic germ
\[
F : \left(  \mathbbm{C}^{n} \times\mathbbm{C}^{\mu+ n - 1}, 0 \right)
\rightarrow\left(  \mathbbm{C}, 0 \right)
\]
defined as
\[
F \left(  z, u \right)  := u_{1} g_{1} \left(  z \right)  + \ldots+ u_{\mu+ n
- 1} g_{\mu+ n - 1} \left(  z \right)  + f_{0} \left(  z \right)
\]
is a versal unfolding of $f_{0}$.
\end{theorem}

{\remark{The proof of the above theorem runs in a very similar way to that
given by Ebeling ({\cite[Prop. 3.17]{Ebe07}}); see also {\cite[Thm.
3.4]{Wal81}} for a more general, but less explicit, approach to the concept of
a versal unfolding and a proof of Theorem \ref{Th Ebe}.}}{\medskip}

Let $f : \left(  \mathbbm{C}^{n}, 0 \right)  \rightarrow\left(  \mathbbm{C}^{}%
, 0 \right)  $ and $g : \left(  \mathbbm{C}^{m}, 0 \right)  \rightarrow\left(
\mathbbm{C}, 0 \right)  $ be two germs of holomorphic functions. We say that
{\emph{$f$ is stably equivalent to $g$\/}} (see {\cite{AGV85}}) iff there
exists $p \in\mathbbm{N}, p \geqslant\max\left(  m, n \right)  $, such that
$\widetilde{f} := f \left(  z_{1}, \ldots, z_{n} \right)  + z_{n + 1}^{2} +
\ldots+ z_{p}^{2}$ is biholomorphically equivalent to $\widetilde{g} := g
\left(  w_{1}, \ldots, w_{m} \right)  + w_{m + 1}^{2} + \ldots+ w_{p}^{2}$,
i.e. there exists a biholomorphism $\Phi: \left(  \mathbbm{C}^{p}, 0 \right)
\rightarrow\left(  \mathbbm{C}^{p}, 0 \right)  $ such that $\widetilde{f}
\circ\Phi= \widetilde{g} $.

It is easy to check that the Milnor number of a singularity is an invariant of the 
stable equivalence. The same is true for the jump of a singularity.

\begin{proposition}
\label{Pro 3}The jump of a singularity is an invariant of the stable equivalence.
\end{proposition}

\begin{proof}
Since obviously $\lambda \left( f \right) = \lambda \left( g \right)$ for
any two biholomorphically equivalent singularities $f$ and $g$, it suffices
to prove that for a singularity $f_0 : \left( \mathbbm{C}^n, 0 \right)
\rightarrow \left( \mathbbm{C}^{}, 0 \right)$ the equality
\begin{equation}
\lambda \left( f_0 \left( z \right) \right) =
\lambda \left( f_0 \left( z \right) + z_{n + 1}^2 \right) \label{Numer*}
\end{equation}
holds, where $z = \left( z_1, \ldots, z_n \right)$.

First we consider the case $\mu(f_0)=1$. Clearly, $\tmop{ord} f_0 = 2$. For the
deformation $f_s(z) \assign f_0(z) + s z_1$ we have $\mu(f_0)-\mu(f_s)=1$, $s \neq 0$. Hence
$\lambda (f_0)=1$. Similarly, $\lambda (f_0(z) + z_{n+1}^2) = 1$.

Now assume that $\mu (f_0) \geqslant 2$. 

First note, that if $\left( f_s \right)$ is a deformation of $f_0$ then the
family $\left( f_s \left( z \right) + z_{n + 1}^2 \right)$ is a deformation
of $f_0 \left( z \right) + z_{n + 1}^2$. Clearly, $\mu\left( f_s \left( z
\right) + z^2_{n + 1} \right) = \mu \left( f_s \left( z \right) \right)$ so
\[ \lambda \left( f_0 \left( z \right) \right) \geqslant \lambda \left( f_0
\left( z \right) + z^2_{n + 1} \right) . \]
	
To prove the opposite inequality we take a deformation $\left( g_s \right)$
of $g_0 \left( z, z_{n + 1} \right) \assign f_0 \left( z \right) + z_{n +
1}^2$ that realizes $\lambda \left( g_0 \right)$ i.e. $\mu \left( g_0
\right) - \mu \left( g_s \right) = \lambda \left( g_0 \right)$ for $s \neq
0$. Let, by Theorem \ref{Th Ebe}, $h_1, \ldots, h_{\mu + n - 1} \in
\mathcal{O}^n$ constitute a basis of $\frac{\mathfrak{m}_n}{\mathfrak{m}_n
\left( \nabla f_0 \right) \mathcal{O}^n}$, where $\mu \assign \mu \left( f_0
\right)$ and $\mathfrak{m}_n$ is the maximal ideal of $\mathcal{O}^n$. Then
$h_1, \ldots, h_{\mu + n - 1}, z_{n + 1}$ constitute a basis of
$\frac{\mathfrak{m}_{n + 1}}{\mathfrak{m}_{n + 1}  \left( \nabla g_0 \right)
\mathcal{O}^{n + 1}}$. Hence, up to a biholomorphism, we may assume that
\[ g_s \left( z, z_{n + 1} \right) = v_1 \left( s \right) h_1 \left( z
\right) + \ldots + v_{\mu + n - 1} \left( s \right) h_{\mu + n - 1}
\left( z \right) + v_{\mu + n} \left( s \right) z_{n + 1} + f_0 \left( z
\right) + z_{n + 1}^2, \]
for holomorphic $v_1, \ldots, v_{\mu + n} : \left( \mathbbm{C}, 0 \right)
\rightarrow \left( \mathbbm{C}, 0 \right)$.

We claim that the $g_s$'es are not smooth. Indeed, in the opposite case we would have
for $s \neq 0$
\begin{equation}
\lambda (g_0) = \mu (g_0) - \mu (g_s) = \mu (g_0) \label{lab}.
\end{equation}
On the other hand, for the deformation $\widetilde{g}_s(z,z_{n+1}):=s (z_1^2+\ldots+z_n^2)+g_0 (z,z_{n+1})$
of $g_0$ we would have, for sufficiently small $s \neq 0$, $\mu(\widetilde{g}_s)=1$ and then
$\mu(g_0)-\mu(\widetilde{g}_s)=\mu(f_0)-1>0$. Hence 
$\lambda(g_0) \leqslant \mu(g_0) - \mu(\widetilde{g}_s) = \mu(g_0)-1$,
a contradiction to (\ref{lab}).

Since the $g_s$'es are not smooth, $\nu_{\mu+n}=0$. Thus for the deformation
\[ f_s \left( z \right) \assign v_1 \left( s \right) h_1 \left( z \right) +
\ldots + v_{\mu + n - 1} \left( s \right) h_{\mu + n - 1} \left( z
\right) + f_0 \left( z \right) \]
of $f_0$ there is $\mu \left( g_s \right) = \mu \left( f_s \right)$ and
\[ \lambda \left( g_0 \right) = \mu \left( g_0 \right) - \mu \left( g_s
\right) = \mu \left( f_0 \right) - \mu \left( f_s \right) = \lambda
\left( \left( f_s \right) \right) . \]
This implies $\lambda \left( f_0 \right) \leqslant \lambda \left( g_0
\right)$.
\end{proof}

\section{Main Results}

In this section we will present the proofs of the results. We begin with the main theorem,
concerning the class $X_9$.

\begin{theorem}
\label{Th1}For the singularities
\begin{equation}
f^{a}_{0} \left(  x, y \right)  = x^{4} + y^{4} + ax^{2} y^{2},
\label{Numer 2.1}%
\end{equation}
{\noindent}where $a \in\mathbbm{C}, a^{2} \neq4$, we have
\[
\lambda\left(  f^{a}_{0} \right)  = 2 \text{.}%
\]
{\noindent}Moreover, for every singularity of type $X_{9}$ its jump is equal
to $2$.
\end{theorem}

First we state and prove a lemma.

\begin{lemma}
\label{Le1}The (classes of the) monomials $x^{i} y^{j}$ with $0 < i + j \leqslant3$ and the
monomial $x^{2} y^{2}$ form a basis of the $\mathbbm{C}$-vector space $\mathfrak{m}/ \left.
\mathfrak{m} \left(  \nabla f^{a}_{0} \right)  \right. $.
\end{lemma}

\begin{proof}
We have $\nabla f^a_0 \left( x, y \right) = \left( 4 x^3 + 2 axy^2,
4 y^3 + 2 ax^2 y \right)$. Let us note that $x^5, x^3 y \in \mathfrak{m} \left( \nabla
f^a_0 \right)$ because
\[ x^5 = \left( \frac{x^2}{4} + \frac{2 ay^2}{4 \left( a^2 - 4 \right)}
\right)  \frac{\partial f^a_0}{\partial x} + \left( \frac{- a^2 xy}{4
\left( a^2 - 4 \right)} \right)  \frac{\partial f^a_0}{\partial y} \]
and
\[ x^3 y = \left( \frac{- y}{\left( a^2 - 4 \right)} \right)  \frac{\partial
f^a_0}{\partial x} + \left( \frac{ax}{2 \left( a^2 - 4 \right)} \right)
\frac{\partial f^a_0}{\partial y} \text{.} \]
Since $f^a_0$ is symmetric with respect to $x$ and $y$, also $y^5, xy^3 \in
\left. \mathfrak{m} \left( \nabla f^a_0 \right) \right.$.
Hence the classes of the monomials
\[x,y,x^2,x y,y^2,x^3,x^2 y, x y^2 ,y^3,x^4,x^2 y^2,y^4 \]
generate $\mathfrak{m}/ \left.\mathfrak{m} \left(  \nabla f^{a}_{0} \right)  \right.$.
Since $x^4 \equiv - {\nobreak} \tfrac{a}{2} x^2 y^2$, $y^4 \equiv - {\nobreak} \tfrac{a}{2} x^2 y^2$
modulo $\left. \mathfrak{m} \left( \nabla f^a_0 \right) \right.$, we get that the classes of the monomials
$x^{i} y^{j}$ with $0 < i + j \leqslant3$ and the monomial $x^{2} y^{2}$
also generate the space $\mathfrak{m}/ \left.\mathfrak{m} \left(  \nabla f^{a}_{0} \right)  \right.$.
They form a basis of $\mathfrak{m}/ \left.\mathfrak{m} \left(  \nabla f^{a}_{0} \right)  \right.$
because by Proposition \ref{Pro 2} $\dim_{\mathbbm{C}} \mathfrak{m}^{} / \left.
\mathfrak{m} \left( \nabla f^a_0 \right) \right.=\dim_{\mathbbm{C}} \mathcal{O}^n / \left.
\mathfrak{m} \left( \nabla f^a_0 \right) \right. -1 =
\dim_{\mathbbm{C}} \mathcal{O}^n / \left.\left( \nabla f^a_0 \right) \mathcal{O}^n \right. +1 =
\mu(f^a_0)+1=10$.
\end{proof}

\begin{proof}[Proof of Theorem \ref{Th1}]
By Proposition \ref{Pro 3} it is enough to prove the first part of the theorem.
Let us fix $a \in \mathbbm{C}, a^2 \neq 4$ and let $f_0 \assign f^a_0$. We
have $\mu \left( f_0 \right) = 9$. Let us consider the
deformation
\[ f_s \left( x, y \right) \assign x^4 + (y^2 + sx)^2 + ax^2 (y^2 + sx)
 \]
of $f_0$. Let us apply the change of
coordinates: $x \mapsto x - sy^2$, $y \mapsto sy$, for $s \neq 0$. In
these coordinates the $f_s$'es take the form
\[ \bar{f}_s \left( x, y \right) = s^2 x^2 + as^3 xy^4 + s^4 y^8 + \left[
asx^3 + x^4 - 2 as^2 x^2 y^2 - 4 sx^3 y^2 + 6 s^2 x^2 y^4 - 4 s^3 xy^6
\right] . \]
It is easily seen that such $\bar{f}_s$'es are non-degenerate if $s \neq 0$.
Thus, by Kouchnirenko theorem, we get $\mu \left(
\bar{f}_s \right) = \nu \left( \bar{f}_s \right) = 7$ and so
\begin{equation}
\mu \left( f_s \right) = 7 \text{ for $s \neq 0$.} \label{Numer 0.25}
\end{equation}
This means that $\lambda \left( \left( f_s \right) \right) = 2$ and therefore
$\lambda \left( f_0 \right) \leqslant 2$. By the definition of the jump of a
singularity, there are only two cases: $\lambda \left( f_0 \right) = 1$ or
$\lambda \left( f_0 \right) = 2$. We will exclude the first possibility.
Suppose to the contrary, that there exists a deformation $\left( f_s
\right)$ of the singularity $f_0$ with the property that
\begin{equation}
\mu \left( f_s \right) = 8 \text{ for } s \neq 0 \text{.} \label{Numer
0.5}
\end{equation}
By Theorem \ref{Th Ebe} and Lemma \ref{Le1} we may assume that
\begin{multline*}
f_s \left( x, y \right)  =  s_{1 \nocomma 0}\left( s \right) x + s_{0
\nocomma 1} \left( s \right) y + s_{2 \nocomma 0} \left( s \right) x^2 +
s_{1 \nocomma 1} \left( s \right) xy + s_{0 \nocomma 2} \left( s \right)
y^2 + s_{30} \left( s \right) x^3 + s_{2 \nocomma 1} \left( s \right) x^2
y \\
 +  s_{1 \nocomma 2} \left( s \right) xy^2 + s_{03} \left( s \right)
y^3 + s_{2 \nocomma 2} \left( s \right) x^2 y^2 + f_0 \left( x, y \right),
\end{multline*}
where $s_{1 \nocomma 0}, \ldots, s_{2 \nocomma 2} : \left( \mathbbm{C}, 0
\right) \rightarrow \left( \mathbbm{C}, 0 \right)$ are holomorphic.
Since $\deg f_s = 4$ and $\mu \left( f_s \right) = 8$ for $s \neq 0$, 
by Theorem \ref{Th GP} three or four of the irreducible components of the curve
$\mathcal{C}_s \assign \left\{ \left( x, y \right) \in \mathbbm{C}^2 : f_s
\left( x, y \right) = 0 \right\}$ pass through the origin. Hence $\tmop{ord}
f_s = 3$ or $\tmop{ord} f_s = 4$, for $0<\left|  s \right|  \ll1$. The latter case is
impossible by Theorem \ref{Th Kouch} because then $\mu \left( f_s \right)
\geqslant \nu \left( f_s \right) \geqslant 9$. Thus, it suffices to consider
the case $\tmop{ord} f_s = 3$.
So, assume $\tmop{ord} f_s = 3$ for $s \neq 0$. Fix any small $s_0 \in
\mathbbm{C} \setminus \left\{ 0 \right\}$. We can write
\[ f_{s_0} \left( x, y \right) = s_{30} x^3 + s_{2 \nocomma 1} x^2 y + s_{1
\nocomma 2} xy^2 + s_{03} y^3 + \left( s_{2 \nocomma 2} + a \right) x^2
y^2 + x^4 + y^4 \text{,} \]
with $s_{i \nocomma j} = s_{i \nocomma j} \left( s_0 \right) \in \mathbbm{C}$.
Since $\tmop{ord} f_{s_0} = 3$, $f_{s_0}$ has to be degenerate. Otherwise, by checking all
the possible cases, we would get $\mu \left( f_{s_0} \right)\leqslant 6$ (by the 
Kouchnirenko theorem), which contradicts (\ref{Numer 0.5}). Since $\tmop{gcd} \left( 3,4 \right)=1$, the degeneracy of $f_{s_0}$
may only happen on a segment of $\Gamma\left(  f_{s_0}\right)$ lying in the line: $u+v=3$.
So, we may write
\[ f_{s_0} \left( x, y \right) = \left( \alpha x + \beta y \right)^2  \left(
\gamma x + \delta y \right) + \left( s_{2 \nocomma 2} + a \right) x^2 y^2
+ x^4 + y^4 \text{,} \]
for some $\alpha, \beta, \gamma, \delta \in \mathbbm{C}$, $\left| \alpha \right|+\left| \beta \right|>0$,
$\left| \gamma \right|+\left| \delta \right|>0$. Moreover, $\alpha \neq 0$ and $\beta\neq 0$
because otherwise $ f_{s_0}$ would be non-degenerate. If we change coordinates:
$\alpha x + \beta y \mapsto x$, $y \mapsto y$ then $f_{s_0}$ takes the form
\[ \widetilde{f}_{s_0} \left( x, y \right) = x^2 \left(
\varepsilon x + \zeta y \right) + P_4 \left (x, y \right) \text{,} \]
where $\varepsilon, \zeta \in \mathbbm{C}$, $\left| \varepsilon \right|+\left| \zeta \right|>0$,
and $P_4$ is a non-zero homogeneous polynomial of degree $4$. We easily check, considering all the 
possible cases, that $\widetilde{f}_{s_0}$ is non-degenerate. So, again by the Kouchnirenko theorem, we
would have
\[ \mu \left( f_{s_0} \right) = \mu ( \widetilde{f}_{s_0} ) = \nu ( \widetilde{f}_{s_0} ) \leqslant 6\text{,}\]
which contradicts (\ref{Numer 0.5}).
\end{proof}

Now we prove a partial result concerning the class $W_{1,0}$.

\begin{proposition}
\label{Th0}For the singularities $f_{0}^{\left(0,b\right)} \left(  x, y \right)  = x^{4} +
y^{6} + bx^{2} y^{4}$, where $b \in\mathbbm{C}$, we have
\[
\lambda(  f_{0}^{\left(0,b\right)} )  = 1 \text{.}%
\]
{\noindent}In particular, $\lambda\left(  x^{4} + y^{6} \right)  = 1$.
\end{proposition}

\begin{proof}
Fix $b \in \mathbbm{C}$. Since $f_{0}^{\left(0,b\right)}$ is Kouchnirenko
non-degenerate, it follows that $\mu ( f_{0}^{\left(0,b\right)} ) = \nu ( f_{0}^{\left(0,b\right)} ) = 15$. Consider the
following deformation of $f_{0}^{\left(0,b\right)}$:
\[ f^{\left(0,b\right)}_s \left( x, y \right) \assign x^4 + \left( y^2 + sx \right)^3 + bx^2
y^4 \text{.} \]
The deformation consists of degenerate singularities (for $s \neq 0$). Apply
the following change of coordinates: $x \mapsto x - sy^2 \comma y \mapsto
sy$. In these coordinates the $f^{\left(0,b\right)}_s$ take the form
\[ \bar{f}^{\left(0,b\right)}_s \left( x, y \right) = s^3 x^3 + (s^4 + bs^6) y^8 + \left[ x^4 -
4 sx^3 y^2 + (6 s^2 + bs^4) x^2 y^4 - (4 s^3 + 2 bs^5) xy^6 \right]
\text{.} \]
It is immediately seen that for $s \neq 0$ the singularities $\bar{f}^{\left(0,b\right)}_s$ are
non-degenerate and so
\[ \mu ( \bar{f}^{\left(0,b\right)}_s ) = 14 \text{.} \]
Since the Milnor number is a biholomorphic (and even a topological)
invariant of a singularity, there is also
\[ \mu ( f^{\left(0,b\right)}_s ) = 14 \text{.} \]
It means that for this particular deformation $( f^{\left(0,b\right)}_s )$ of $f_{0}^{\left(0,b\right)}$
we have $\lambda (( f^{\left(0,b\right)}_s )) = 1$ and consequently
$\lambda( f_{0}^{\left(0,b\right)}) = 1$.
\end{proof}

\begin{corollary}
For every singularity $f_0$ stably equivalent to one of $f^{\left(0,b\right)}_0$, $b\in\mathbbm{C}$,
the jump $\lambda(f_0)$ of $f_0$ is equal to $1$.
\end{corollary}


Proposition {\ref{Th0}} implies also that $\lambda(f_0)$ is not a topological invariant of $f_0$.
Recall that two singularities $f$ and $g$ in $\mathbbm{C}^n$ have the same topological type if there
exist neighbourhoods $U$ and $V$ of $0\in\mathbbm{C}^n$ and a homeomorphism $\Phi : U \rightarrow V$
such that $\Phi(V(f))=V(g)$, where $V(f)$ (resp. $V(g)$) is the zero set of $f$ (resp. $g$) in $U$
(resp. $V$).

\begin{corollary}\label{cor2}
The jump of the Milnor number $\lambda(f_0)$ is not a topological invariant of $f_0$.
\end{corollary}

\begin{proof}
By the Gusein-Zade theorem, for generic elements $f_{0}^{\left ( a,b \right )}$ of the class $W_{1,0}$
we have $\lambda(f_{0}^{\left ( a,b \right )})>1$. Proposition {\ref{Th0}} gives that the elements $f_{0}^{\left ( 0,b \right )}$
of $W_{1,0}$ satisfy $\lambda(f_{0}^{\left ( 0,b \right )})=1$. But all the singularities $f_{0}^{\left ( a,b \right )}$,
$a,b\in\mathbbm{C}$, $a^2\neq4$, have the same topological type. This follows from the general L\^e-Ramanujam theorem on $\mu$-constant
families of singularities or from the (much easier) fact that all the singularities $f_{0}^{\left ( a,b \right )}$ have the same
resolution graph.
\end{proof}
{\medskip}
\begin{acknowledgement}
We thank prof. A. P\l oski for discussions which led to improvement of the text of the paper.
\end{acknowledgement}

\end{document}